\documentclass[9pt]{extarticle} 
\usepackage[ 
  paperwidth=5.88in,paperheight=8.60in,
  textwidth=4.2in,textheight=6.7in]{geometry} 

\usepackage{fancyhdr} 
\usepackage{amsmath, amsthm, amsfonts, amssymb, mathdots}
\usepackage{stmaryrd}
\usepackage{paralist}
\usepackage{mathtools}
\usepackage{tikz}
\usepackage{graphicx}
\DeclareMathOperator*{\argmin}{arg\,min}
\usepackage{cleveref}
\newtheorem{theorem}{Theorem}
\newtheorem{definition}{Definition}
\newtheorem{lemma}{Lemma}
\newtheorem{corollary}{Corollary}
\newtheorem{proposition}{Proposition}
\theoremstyle{definition}

\fancypagestyle{headings}{%
\fancyhf{} 
\fancyfoot[L]{
}

}

\makeatletter 
\renewcommand{\maketitle}{
\pagestyle{plain}
\vspace*{\baselineskip}
\begin{center}
\MakeUppercase{\small Philip Saltenberger} \\ \emph{Institute for Numerical Analysis, TU Braunschweig} \\ (E-Mail: philip.saltenberger@tu-bs.de) \\[0.25cm]
\MakeUppercase{\small Michel-Niklas Senn} \\ \emph{Institute for Numerical Analysis, TU Braunschweig} \\ (E-Mail: m.senn@tu-braunschweig.de)

\vspace*{3\baselineskip}

\MakeUppercase{\large \@title}
\end{center} \vskip-\baselineskip
}
\makeatother

\usepackage{titlesec} 
\titleformat{\section}[hang]{\scshape}{\thesection. }{0pt}{\centering}[]

\title{On skew-Hamiltonian Matrices and their
Krylov-Lagrangian Subspaces}

\begin{document}
\maketitle

\begin{abstract}
It is a well-known fact that the Krylov space $\mathcal{K}_j(H,x)$ generated by a skew-Hamiltonian matrix $H \in \mathbb{R}^{2n \times 2n}$ and some $x \in \mathbb{R}^{2n}$ is isotropic for any $j \in \mathbb{N}$.
For any given isotropic subspace $\mathcal{L} \subset \mathbb{R}^{2n}$ of dimension $n$---which is called a Lagrangian subspace---the question whether $\mathcal{L}$ can be generated as the Krylov space of some skew-Hamiltonian matrix is considered. The affine variety $\mathbb{HK}$ of all skew-Hamiltonian matrices $H \in \mathbb{R}^{2n \times 2n}$ that generate $\mathcal{L}$ as a Krylov space is analyzed. Existence and uniqueness results are proven, the dimension of $\mathbb{HK}$ is found and skew-Hamiltonian matrices with minimal $2$-norm, minimal Frobenius norm and prescribed eigenvalues in $\mathbb{HK}$ are identified. Some applications of the presented results are given.
\end{abstract}

\section{Introduction} \label{intro-sec}

This work establishes a link between three important and well-known concepts and tools from linear algebra: Krylov spaces, isotropic subspaces related to a special bilinear form and skew-Hamiltonian matrices. Before we start, we review all three concepts briefly. To set the stage, let
\begin{equation} J_{2n} = \begin{bmatrix} 0 & I_n \\ -I_n & 0 \end{bmatrix} \in \mathbb{R}^{2n \times 2n} \label{equ:J} \end{equation}
and consider the bilinear form $[x,y]_{J_{2n}} :=x^TJ_{2n}y$ on $\mathbb{R}^{2n} \times \mathbb{R}^{2n}$. 
The two main concepts related to the bilinear form $[ \cdot, \cdot]_{J_{2n}}$ we consider here are isotropic subspaces and skew-Hamiltonian operators.

A subspace $\mathcal{L} \subset \mathbb{R}^{2n}$ such that $[x,y]_{J_{2n}}=0$ holds for any $x,y \in \mathcal{L}$ is called isotropic. Such subspaces arise from the indefiniteness of $J_{2n}$ (which has eigenvalues $+ \imath$ and $- \imath$) and have no analogue for symmetric positive definite inner products. It can be shown that the maximum possible dimension of an isotropic subspace $\mathcal{L} \subset \mathbb{R}^{2n}$ is $n$. In this case $\mathcal{L}$ is called Lagrangian subspace \cite[Def.\,6]{Benner05}. Lagrangian subspaces are well-studied and play a crucial role for structured decompositions of (skew)-Hamiltonian matrices \cite{PVL81} or the solution of algebraic Riccati equations \cite{Abou03}.

A matrix $H \in \mathbb{R}^{2n \times 2n}$ is called skew-Hamiltonian if $J_{2n}^TH^TJ_{2n} = H$ holds. Along with Hamiltonian matrices which satisfy $J_{2n}^TH^TJ_{2n}=-H$ these matrices arise frequently in systems and control theory \cite{Benner05} or quadratic eigenvalue problems \cite{Tiss01}. Skew-Hamiltonian matrices can be interpreted as adjoint operators with respect to the bilinear form $[\cdot, \cdot]_{J_{2n}}$. In fact, a skew-Hamiltonian matrix $H \in \mathbb{R}^{2n \times 2n}$ satisfies $[x,Hy]_{J_{2n}} = [Hx,y]_{J_{2n}}$ for all $x,y \in \mathbb{R}^{2n}$. On the other hand, a Hamiltonian matrix $H$ always satisfies $[x,Hy]_{J_{2n}} = [-Hx,y]_{J_{2n}}$ and therefore represents a skew-adjoint operator with respect to $[\cdot, \cdot]_{J_{2n}}$.

Here, we are mainly interested in Krylov subspaces of skew-Hamiltonian matrices.
Given some matrix $A \in \mathbb{R}^{2n \times 2n}$ and some vector $x \in \mathbb{R}^{2n}$ the $j$-th Krylov subspace for $A$ and $x$ is defined as
$$ \mathcal{K}_j(A,x) = \textnormal{span} \lbrace x,Ax,A^2x, \ldots , A^{j-1}x \rbrace.$$
We sometimes refer to $(x,Ax, \ldots , A^{j-1}x)$ as a Krylov sequence for $A$ and $x$. Krylov subspaces are a very important tool in numerical linear algebra since a large amount of the most efficient algorithms for solving eigenvalue problems and linear systems is based on Krylov subspaces \cite{Vorst, Liesen}.

The fundamental relationship between all three concepts and definitions is given in the next proposition.
It is the starting point of our investigations and can be found with its proof in \cite[Prop.\,3.3]{MehrWat}.

\begin{proposition} \label{prop:Krylov-is-Lagrangian}
Let $H \in \mathbb{R}^{2n \times 2n}$ be skew-Hamiltonian and $x \in \mathbb{R}^{2n}$. Then the Krylov subspace $\mathcal{K}_j(H,x)$ is isotropic for any $j \in \mathbb{N}$.
\end{proposition}

In other words, ``the beauty of skew-Hamiltonian operators is that the Krylov subspaces that they generate are isotropic'' \cite[p.\,1910]{MehrWat}. This fact is exploited in the SHIRA algorithm introduced in \cite{MehrWat} for the solution of Hamiltonian eigenvalue problems and its variants, see \cite{BennEff, MehrWat1}.

In this work we turn the statement of Proposition \ref{prop:Krylov-is-Lagrangian} upside down. That is, we consider the question whether every isotropic subspace $\mathcal{L} \subset \mathbb{R}^{2n}$ is the Krylov space of some skew-Hamiltonian matrix $H \in \mathbb{R}^{2n \times 2n}$ and answer it positively. Additionally, we characterize the set of all such skew-Hamiltonian matrices $H$ for which a given isotropic subspace $\mathcal{L}$ arises as a Krylov space. Finally, we identify elements in this set with special additional properties. With all our investigations we focus on Lagrangian subspaces, i.e., isotropic subspaces of dimension $n$. Whenever some Lagrangian subspace $\mathcal{L} \subset \mathbb{R}^{2n}$ arises as a Krylov subspace of a skew-Hamiltonian matrix $H \in \mathbb{R}^{2n \times 2n}$, there is a sequence of vectors $x_1, x_2 = Hx_1, x_3=H^2x_1, \ldots , x_n = H^{n-1}x_1$ such that $$\mathcal{L} = \mathcal{K}_n(H,x_1) = \textnormal{span} \lbrace x_1,x_2, \ldots , x_n \rbrace.$$ This brings us to the following definition.

\begin{definition} \label{def:realizing-set}
Let $\mathcal{L} \subset \mathbb{R}^{2n}$ be a Lagrangian subspace with a given ordered basis $\mathbf{B} = (x_1, \ldots , x_n)$.
We say that a skew-Hamiltonian matrix $H \in \mathbb{R}^{2n \times 2n}$ realizes $\mathcal{L}$ via $\mathbf{B}$ as a Krylov subspace, if the relation
\begin{equation} Hx_k = x_{k+1} \label{equ:Krylov-relations} \end{equation} holds for all $k=1, \ldots , n-1$. The set of all skew-Hamiltonian matrices $H \in \mathbb{R}^{2n \times 2n}$ satisfying \eqref{equ:Krylov-relations} is denoted either by $\mathbb{HK}(x_1, \ldots , x_n)$ or $\mathbb{HK}(\mathbf{B})$.
\end{definition}

The main purpose of this work is the thorough analysis of the set $\mathbb{HK}(\mathbf{B})$ of skew-Hamiltonian matrices $H \in \mathbb{R}^{2n \times 2n}$ that generate a given Lagrangian subspace $\mathcal{L} \subset \mathbb{R}^{2n}$ as a Krylov space $\mathcal{K}_n(H,x)$ for some $x \in \mathcal{L}$ resulting in the Krylov sequence $\mathbf{B}=(x,Hx,\ldots , H^{n-1}x)$. The task we consider is related to the technique of dynamic mode decomposition (DMD) in model reduction. In DMD a given set of snapshots $x_1, \ldots, x_n$ is assumed to be correlated via $Ax_k = x_{k+1}$, $k=1, \ldots , n-1,$ for a matrix $A$ and one is interested in finding eigenvalues and eigenvectors of $A$ without determining $A$ first \cite{dmd}. In contrast to that, our goal is to determine all those matrices $A$ explicitly under the additional assumptions that $\textnormal{span}\lbrace x_1, \ldots , x_n \rbrace$ is a Lagrangian subspace and $A$ is a skew-Hamiltonian matrix.

This work is structured as follows:

\begin{itemize}
\item[(a)] In Section \ref{sec:construction} we  show how to construct a particular skew-Hamiltonian matrix $\widehat{H}$ in $\mathbb{HK}(\mathbf{B})$ (Theorem \ref{thm:main-existence}) and use this result to derive an explicit pa\-ra\-met\-ri\-za\-tion of $\mathbb{HK}(\mathbf{B})$ (Theorem \ref{thm:complete_HK}). In consequence, for any ordered basis $\mathbf{B}$ of $\mathcal{L}$ the set $\mathbb{HK}(\mathbf{B})$ is proven to be nonempty. Moreover, we show that $\mathbb{HK}(x_1, \ldots , x_n)$ is an affine subspace of the vector space of all skew-Hamiltonian matrices in $\mathbb{R}^{2n \times 2n}$ and prove that its dimension is $n(n+1)/2$ (Corollary \ref{thm:main-dimension}).
\item[(b)] In Section \ref{sec:minimum-norm} we prove that the particular matrix $\widehat{H} \in \mathbb{HK}(x_1, \ldots , x_n)$ constructed in Theorem \ref{thm:main-existence} is very special since it has minimal $2$-norm and minimal Frobenius norm among all matrices in $\mathbb{HK}(\mathbf{B})$ (Theorems \ref{thm:leastnorm} and \ref{thm:leastnorm2}). Since a Lagrangian subspace $\mathcal{L} \subset \mathbb{R}^{2n}$ arising as a Krylov space of some skew-Hamiltonian matrix $H \in \mathbb{R}^{2n \times 2n}$ is always $H$-invariant, we additionally characterize those matrices $H \in \mathbb{HK}(\mathbf{B})$ where the spectrum of $H|_{\mathcal{L}}$ coincides with a given set of $n$ values.
\item[(c)] In Section \ref{sec:applications} we consider a scenario where the results from Section \ref{sec:minimum-norm} can be applied. Whenever $H$ is a perturbed skew-Hamiltonian matrix, i.e. $H \neq J_{2n}^TH^TJ_{2n}$, the Krylov space $\mathcal{K}_n(H, x)$ will in general neither be isotropic nor invariant. Thus, we approximate the Krylov sequence $(x,Hx, \ldots , H^{n-1}x)$ by a sequence of vectors $\mathbf{B} = (y_1, \ldots , y_n)$ spanning a Lagrangian subspace and we determine the best possible skew-Hamiltonian approximation $H' \in \mathbb{HK}(\mathbf{B})$  to $H$.
\end{itemize}

\subsection{Notation and preliminary results}
\label{sec:preliminaries}

Throughout the subsequent sections we use the short-hand-notation $H^\star$ to denote $J_{2n}^TH^TJ_{2n}$ for any $H \in \mathbb{R}^{2n \times 2n}$. In the context of the indefinite inner product $[x,y]_{J_{2n}} = x^TJ_{2n}y$, the matrix $H^\star $ is the adjoint of $H$. In other words, $H \in \mathbb{R}^{2n \times 2n}$ is skew-Hamiltonian if and only if $H = H^\star$. Moreover, notice that $(H^\star)^\star = H$ holds for any arbitrary matrix $H$. We write $\textnormal{span} \lbrace z_1, \ldots , z_k \rbrace$ to denote the subspace spanned by the vectors $z_1, \ldots , z_k$ and use the notation $\mathcal{S}^{\bot}$ for the orthogonal complement of a subspace $S \subset \mathbb{R}^{2n}$. If $A \in \mathbb{R}^{m \times n}$ we denote the Moore-Penrose pseudoinverse by $A^+$. The subspace spanned by the columns of $A$ is denoted by $\textnormal{range}(A)$.

There exists a direct relationship between skew-Hamiltonian and skew-symmetric matrices that is stated in Proposition \ref{prop:skew} and which is used several times throughout this work.

\begin{proposition} \label{prop:skew}
For any skew-symmetric matrix $S \in \mathbb{R}^{2n \times 2n}$ the matrix $H := J_{2n}^TS$ is skew-Hamiltonian. Moreover, for any skew-Hamiltonian matrix $H \in \mathbb{R}^{2n \times 2n}$ there is a unique skew-symmetric matrix $S$ such that $H = J_{2n}^TS$.
\end{proposition}

\begin{proof}
Let $H \in \mathbb{R}^{2n \times 2n}$ be skew-Hamiltonian. Then, as $(J_{2n}H)^T = -J_{2n}H$ holds, the matrix $S := J_{2n}H$ is skew-symmetric. Now $H$ can be expressed as $J_{2n}^TS$ where $S$ is skew-symmetric. On the other hand, if $S$ is skew-symmetric and $H := J_{2n}^TS$, then $(J_{2n}H)^T = S^T = -S = -(J_{2n}H)$. Therefore, $H$ is skew-Hamiltonian.
\end{proof}

As $J_{2n}$ is nonsingular, the mapping $A \mapsto J_{2n}^TA$
provides an isomorphism between the vector subspaces of skew-symmetric and skew-Hamiltonian matrices in $\mathbb{R}^{2n \times 2n}$. Consequently, both vector spaces have dimension $n(2n-1)$.

\section{Construction and characterization of $\mathbb{HK}(\mathbf{B})$}
\label{sec:construction}

According to Proposition \ref{prop:Krylov-is-Lagrangian}, the Krylov subspace $\mathcal{K}_j(H,x)$, $j \in \mathbb{N}$, of any skew-Hamiltonian matrix $H \in \mathbb{R}^{2n \times 2n}$ and any $x \in \mathbb{R}^{2n}$ is always isotropic. Therefore, $\mathcal{K}_n(H,x)$ is Lagrangian whenever $\dim(\mathcal{K}_n(H,x)) = n$.
The natural inverse question is, if for any given Lagrangian subspace $\mathcal{L} \subset \mathbb{R}^{2n}$ there always exists some skew-Hamiltonian matrix $H \in \mathbb{R}^{2n \times 2n}$ and a basis $\mathbf{B}$ of $\mathcal{L}$ such that $\mathcal{L} = \mathcal{K}_n(H,x)$ where $\mathbf{B}=(x,Hx, \ldots , H^{n-1}x)$. 
This question is answered in Theorem \ref{thm:main-existence} with an explicit construction.

In Theorem \ref{thm:main-existence} and throughout this work we use the following definitions: given some Lagrangian subspace $\mathcal{L} \subset \mathbb{R}^{2n}$ and a basis $\mathbf{B}=(x_1, \ldots , x_n)$ of $\mathcal{L}$ we define
\begin{equation} \begin{aligned}
X_L &:= \begin{bmatrix} x_1 & x_2 & \cdots & x_{n-1} \end{bmatrix} \in \mathbb{R}^{2n \times n-1},\\
X_R &:= \begin{bmatrix} x_2 & x_3 & \cdots & x_n \end{bmatrix} \in \mathbb{R}^{2n \times n-1}.
\end{aligned} \label{equ:XLXR} \end{equation}
With the definitions in \eqref{equ:XLXR} we now have the following result.

\begin{theorem} \label{thm:main-existence}
Let $\mathcal{L} \subset \mathbb{R}^{2n}$ be a Lagrangian subspace and $\mathbf{B} = (x_1, \ldots , x_n)$ any ordered basis of $\mathcal{L}$. Consider the matrix
\begin{equation} \widehat{H} := X_RX_L^+ + J_{2n}^T(X_RX_L^+)^TJ_{2n} \in \mathbb{R}^{2n \times 2n}. \label{equ:H_hat} \end{equation}
Then $\widehat{H} \in \mathbb{HK}(\mathbf{B})$.
\end{theorem}

\begin{proof}
First notice that $\widehat{H}$ is of the form $K + K^\star$ where $K=X_RX_L^+$. Thus $(K+K^\star)^\star = K^\star + K = \widehat{H}$, so $\widehat{H}$ is skew-Hamiltonian.
Now, since $\mathcal{L}$ is Lagrangian, $X_R^TJ_{2n}x_k=0$ holds for all $x_k$, $k=1, \ldots , n$. In particular, for $x_1, \ldots, x_{n-1}$ we obtain
$$ \widehat{H}x_k = \big( X_RX_L^+ + J_{2n}^T(X_L^+)^TX_R^TJ_{2n} \big) x_k = X_RX_L^+x_k = X_Re_k = x_{k+1}$$
Thus we have $\widehat{H}x_k = x_{k+1}$ for all $k=1, \ldots , n-1$ which shows that $\widehat{H} \in \mathbb{HK}(\mathbf{B})$ and completes the proof.
\end{proof}

In other words, Theorem \ref{thm:main-existence} states that $\mathbb{HK}(x_1, \ldots , x_n) \neq \emptyset$ for any set of linearly independent vectors $x_1, \ldots , x_n$ spanning a Lagrangian subspace $\mathcal{L} \subset \mathbb{R}^{2n}$. Since the construction in Theorem \ref{thm:main-existence} works analogously if $\dim(\mathcal{L}) = k < n$ and $\mathbf{B}=(x_1, \ldots , x_k)$, we may now extend Proposition \ref{prop:Krylov-is-Lagrangian} by saying that \emph{a subspace $\mathcal{L} \subset \mathbb{R}^{2n}$ is isotropic if and only if it is the Krylov space of some skew-Hamiltonian matrix}.

Defining $K := X_RX_L^+ \in \mathbb{R}^{2n \times 2n}$ from the matrices in \eqref{equ:XLXR}, recall that $\widehat{H}$ in \eqref{equ:H_hat} can be expressed as $\widehat{H} = K + K^\star$. In view of the bilinear form $[x,y] = x^TJ_{2n}y$, the matrix $K$ has itself a very special property.

\begin{corollary}
Let $\mathcal{L} \subset \mathbb{R}^{2n}$ be a Lagrangian subspace and $\mathbf{B} = (x_1, \ldots , x_n)$ any ordered basis of $\mathcal{L}$.
Then for $K = X_RX_L^+$ we have
$$KK^\star = K^\star K = 0.$$
\end{corollary}

\begin{proof}
Recall that $X_L^TJ_{2n}X_L = X_R^TJ_{2n}X_R = 0$ holds since $\mathcal{L}$ is Lagrangian. The proof now follows from the observation that $X_L^+ = (X_L^TX_L)^{-1}X_L^T$. In particular, for $KK^\star$ we obtain
\begin{align*}
KK^\star &= X_R \left( X_L^TX_L \right)^{-1}X_L^T J_{2n}^TX_L \left( X_L^TX_L \right)^{-1}X_R^T J_{2n} = 0
\end{align*}
since $X_L^TJ_{2n}^TX_L = - X_L^TJ_{2n}X_L = 0$. Similarly, $K^\star K=0$ follows by observing that $X_R^TJ_{2n}X_R = 0$ holds.
\end{proof}

Matrices $K \in \mathbb{R}^{2n \times 2n}$ with the property $KK^\star = K^\star K$ are called $J_{2n}$-normal. This type of matrices was investigated in, e.g., \cite{Mehl06}.

Recall that a matrix $A \in \mathbb{R}^{n \times n}$ is called nilpotent of index $k \in \mathbb{N}$ if $A^{k-1} \neq 0$ and $A^k = 0$ holds. For some orthonormal basis of a Lagrangian subspace we have the following result.

\begin{corollary}
Let $\mathcal{L} \subset \mathbb{R}^{2n}$ be a Lagrangian subspace and $\mathbf{B} = (x_1, \ldots , x_n)$ any ordered orthonormal basis of $\mathcal{L}$. Then the matrix $\widehat{H} \in \mathbb{HK}(\mathbf{B})$ defined in \eqref{equ:H_hat} is nilpotent of index $n$.
\end{corollary}

\begin{proof}
First note that, as $x_1, \ldots , x_n$ are orthogonal, $X_L^+ = X_L^T$. Due to the relations $X_L^TJ_{2n}X_L = 0$ and  $X_R^T J_{2n}X_R = 0$ it is verified straightly forward that
$$ \widehat{H}^k = X_R \left( X_L^TX_R \right)^{k-1} X_L^T + J_{2n}^T X_L \left[ \left( X_L^TX_R \right)^{k-1} \right]^T X_R^TJ_{2n}$$
holds for all $k \in \mathbb{N}$. Due to the orthogonality of $x_1, \ldots , x_n$ we have
$$X_L^TX_R = \begin{bmatrix} 0 & 0 & \cdots & 0 & 0 \\ 1 & 0 & & & 0 \\ 0 & 1 & \ddots & & \vdots \\ \vdots & & \ddots & \ddots & 0 \\ 0 & \cdots & 0 & 1 & 0 \end{bmatrix} \in \mathbb{R}^{n-1 \times n-1}$$
from which it follows directly that $(X_L^TX_R)^{n-1} = 0$. This in turn implies $\widehat{H}^n = 0$ and proves the result.
\end{proof}

The next theorem gives a comprehensive characterization of $\mathbb{HK}(\mathbf{B})$.

\begin{theorem} \label{thm:complete_HK}
Let $\mathcal{L} \subset \mathbb{R}^{2n}$ be a Lagrangian subspace and $\mathbf{B} = (x_1, \ldots , x_n)$ any ordered basis of $\mathcal{L}$. Moreover, let $X_L^{\bot} \in \mathbb{R}^{2n \times n+1}$ be a matrix whose columns are an orthonormal basis of $\textnormal{span} \lbrace x_1, \ldots , x_{n-1} \rbrace^{\bot}$ and let $\widehat{H} \in \mathbb{R}^{2n \times 2n}$ be the matrix defined in \eqref{equ:H_hat}.
Then $H \in \mathbb{HK}(\mathbf{B})$ if and only if it can be expressed as
\begin{equation} H = \widehat{H} + J_{2n}^TX_L^{\bot} S (X_L^{\bot})^T \label{equ:char_HK} \end{equation}
where $S \in \mathbb{R}^{(n+1) \times (n+1)}$ is some skew-symmetric matrix.
\end{theorem}

\begin{proof}
Suppose $H \in \mathbb{HK}(\mathbf{B})$ and define $\Delta H := H - \widehat{H}$. Then $\Delta H \in \mathbb{R}^{2n \times 2n}$ is skew-Hamiltonian and $(\Delta H)x_k = 0$ holds for all $k=1, \ldots, n-1$. According to Proposition \ref{prop:skew} the matrix $\widetilde{S} := J_{2n}(\Delta H)$ is skew-symmetric and $\widetilde{S}x_k = 0$ still holds for all $k=1, \ldots , n-1$. Let $(\tilde{x}_1, \ldots , \tilde{x}_{n-1})$ be some orthonormal basis of $\textnormal{span} \lbrace x_1, \ldots , x_{n-1} \rbrace$ and set $\widetilde{X}_L := [ \, \tilde{x}_1 \; \cdots \; \tilde{x}_{n-1} \; ]$. Then the matrix $Z := [\; X_L^{\bot} \;  \widetilde{X}_L \; ] \in \mathbb{R}^{2n \times 2n}$ is orthogonal and
$$ Z^T\widetilde{S}Z =  \begin{bmatrix} (X_L^{\bot})^T\widetilde{S}X_L^{\bot} & (X_L^{\bot})^T\widetilde{S} \widetilde{X}_L \\ -\widetilde{X}_L^T\widetilde{S}^TX_L^{\bot} & \widetilde{X}_L^T\widetilde{S} \widetilde{X}_L \end{bmatrix} = \begin{bmatrix} (X_L^{\bot})^T\widetilde{S}X_L^{\bot} & 0 \\ 0 & 0 \end{bmatrix}$$
which shows that
\begin{equation} \widetilde{S} = Z \begin{bmatrix} (X_L^{\bot})^T\widetilde{S}X_L^{\bot} & 0 \\ 0 & 0 \end{bmatrix} Z^T = X_L^{\bot} S(X_L^{\bot})^T, \quad \textnormal{with} \; S = (X_L^{\bot})^T\widetilde{S}X_L^{\bot}. \label{equ:S_tilde} \end{equation}
Consequently, since $\Delta H = J_{2n}^T \widetilde{S}$ we obtain
$$H = \widehat{H}+ \Delta H = \widehat{H} + J_{2n}^T X_L^{\bot} S(X_L^{\bot})^T$$ where $S \in \mathbb{R}^{(n+1) \times (n+1)}$ is as defined in \eqref{equ:S_tilde}. Notice that $S$ is is skew-symmetric.
Therefore, $H$ is of the form \eqref{equ:char_HK}. On the other hand, if $H$ is as in \eqref{equ:char_HK} for some $S=-S^T \in \mathbb{R}^{(n+1) \times (n+1)}$, then for $x_1, \ldots , x_{n-1}$ we have
$$Hx_k = \widehat{H}x_k + J_{2n}^TX_L^{\bot} S (X_L^{\bot})^T x_k = \widehat{H}x_k = x_{k+1}.$$
Therefore $H \in \mathbb{HK}(\mathbf{B})$ which completes the proof.
\end{proof}
Defining the matrix subspace
$$ \mathcal{H}_0 := \big\lbrace J_{2n}^TX_L^{\bot} S (X_L^{\bot})^T \; ; \; S = -S^T \in \mathbb{R}^{(n+1) \times (n+1)} \big\rbrace$$
we see that $\mathbb{HK}(\mathbf{B}) = \widehat{H} + \mathcal{H}_0$ is an affine subspace of $\mathbb{R}^{2n \times 2n}$.
The following Corollary \ref{thm:main-dimension} states the dimension of $\mathbb{HK}(\mathbf{B})$ using the common definition $\dim(\mathbb{HK}(\mathbf{B})) = \dim(\mathcal{H}_0)$.

\begin{corollary} \label{thm:main-dimension}
Let $\mathcal{L} \subset \mathbb{R}^{2n}$ be a Lagrangian subspace and $\mathbf{B}= (x_1, \ldots , x_n)$ any ordered basis of $\mathcal{L}$. Then
\begin{equation} \dim \left( \mathbb{HK}(\mathbf{B}) \right) = \frac{n(n+1)}{2}. \label{equ:result-dim} \end{equation} \end{corollary}

\begin{proof}
The space of all $(n+1) \times (n+1)$ skew-symmetric matrices is isomorphic to $\mathcal{H}_0$ via the isomorphism
$ S \mapsto J_{2n}^T X_L^{\bot}S(X_L^{\bot})^T$.
Therefore we have $\dim(\mathbb{HK}(\mathbf{B})) = \dim(\mathcal{H}_0) = n(n+1)/2$ which equals the dimension of the space of all real $(n+1) \times (n+1)$ skew-symmetric matrices.
\end{proof}

In the next section we identify special elements in $\mathbb{HK}(\mathbf{B})$ with respect to their norm and their eigenvalues.

\section{Special elements in $\mathbb{HK}(\mathbf{B})$}
\label{sec:minimum-norm}

Let $\mathcal{L} \subset \mathbb{R}^{2n}$ be some Lagrangian subspace and $\mathbf{B}=(x_1, \ldots , x_n)$ some basis of $\mathcal{L}$. Any matrix $H \in \mathbb{HK}(\mathbf{B})$ generates $\mathcal{L}$ as a Krylov subspace in the sense that
\begin{equation}  \mathcal{L} = \mathcal{K}_n(H,x_1) = \textnormal{span} \lbrace x_1, \ldots , x_n \rbrace, \quad x_k = H^{k-1}x_1, k=1, \ldots , n-1. \label{equ:min_norm_intro} \end{equation}
In this section we identify skew-Hamiltonian matrices $H \in \mathbb{R}^{2n \times 2n}$ that satisfy the relation \eqref{equ:min_norm_intro} with smallest Frobenius norm, smallest 2-norm and a given set of $n$ eigenvalues when $H$ is restricted to ${\mathcal{L}}$.

\subsection{Matrices in $\mathbb{HK}(\mathbf{B})$ with smallest Frobenius norm}
\label{subsec:minimum-norm1}

First we consider the task of minimizing the Frobenius norm on $\mathbb{HK}(\mathbf{B})$.

\begin{theorem} \label{thm:leastnorm}
Let $\mathcal{L} \subset \mathbb{R}^{2n}$ be a Lagrangian subspace and $\mathbf{B} = (x_1, \ldots , x_n)$ any ordered basis of $\mathcal{L}$. Furthermore, let $\widehat{H} \in \mathbb{R}^{2n \times 2n}$ be as defined in \eqref{equ:H_hat}.
 Then 
$$ \Vert \widehat{H} \Vert_F \leq \Vert H \Vert_F \qquad \forall \, H \in \mathbb{HK}(\mathbf{B}).$$
\end{theorem}

\begin{proof}
Let $\widetilde{X}_L \in \mathbb{R}^{2n \times (n-1)}$ be a matrix whose orthonormal columns are a basis for $\textnormal{span} \lbrace x_1, \ldots , x_{n-1} \rbrace$. If $H \in \mathbb{HK}(\mathbf{B})$ is an arbitrary element, $\Vert H \Vert_F$ can be expressed as
\begin{equation}
\left\Vert \widehat{H} + J_{2n}^TX_L^{\bot}S(X_L^{\bot})^T \right\Vert_F = \left\Vert \widehat{H} + J_{2n}^T [ \; X_L^{\bot} \;\; \widetilde{X}_L \; ] \begin{bmatrix} S & 0 \\ 0 & 0 \end{bmatrix} [ \; X_L^{\bot} \;\; \widetilde{X}_L \; ]^T \right\Vert_F.
\end{equation}
As $J_{2n}$ and $[ \; X_L^{\bot} \; \widetilde{X}_L \; ]$ are orthogonal, multiplying by $J_{2n}$ and $[ \; X_L^{\bot} \; \widetilde{X}_L \; ]^T$ from the left and $[ \; X_L^{\bot} \;\; \widetilde{X}_L \; ]$ from the right gives
\begin{equation}
\begin{aligned} \Vert H \Vert_F &= \left\Vert [ \; X_L^{\bot} \;\; \widetilde{X}_L \; ]^T J_{2n} \widehat{H} [ \; X_L^{\bot} \;\; \widetilde{X}_L \; ] + \begin{bmatrix} S & 0 \\ 0 & 0 \end{bmatrix} \right\Vert_F \\ &= \left\Vert \begin{bmatrix} (X_L^{\bot})^T(J_{2n} \widehat{H})X_L^{\bot} & (X_L^{\bot})^T(J_{2n} \widehat{H})\widetilde{X}_L \\ \widetilde{X}_L^T(J_{2n} \widehat{H})X_L^{\bot}  & \widetilde{X}_L^T(J_{2n} \widehat{H}) \widetilde{X}_L \end{bmatrix} + \begin{bmatrix} S & 0 \\ 0 & 0 \end{bmatrix} \right\Vert_F
\end{aligned} \label{equ:minimizeS} \end{equation}
which is obviously minimized if $S = - (X_L^{\bot})^T(J_{2n} \widehat{H})X_L^{\bot}$. Note that $S$ is skew-symmetric. This yields $H = \widehat{H} - J_{2n}^TX_L^{\bot}((X_L^{\bot})^TJ_{2n} \widehat{H}X_L^{\bot})(X_L^{\bot})^T$ as the matrix in $\mathbb{HK}(\mathbf{B})$ with minimal Frobenius norm. Using the properties $X_L^+ = (X_L^TX_L)^{-1}X_L^T$ and $X_L^+X_L^{\bot} = 0$  a direct calculation shows that the matrix $X_L^{\bot}((X_L^{\bot})^TJ_{2n} \widehat{H}X_L^{\bot})(X_L^{\bot})^T$ is the zero matrix. 
Thus we obtain $H = \widehat{H}$ as the matrix with minimal Frobenius norm in $\mathbb{HK}(\mathbf{B})$. 
\end{proof}

In other words, Theorem \eqref{thm:leastnorm} states that
$ \argmin_{H \, \in \, \mathbb{HK}(\mathbf{B})} \Vert H \Vert_F = \widehat{H}$.
In Section \ref{subsec:minimum-norm2} we show that this remains to hold if $\Vert \cdot \Vert_F$ is replaced by $\Vert \cdot \Vert_2$.

\subsection{Matrices in $\mathbb{HK}(\mathbf{B})$ with smallest 2-norm}
\label{subsec:minimum-norm2}

Recall that for any $A \in \mathbb{R}^{n \times n}$ we have
$$ \Vert A \Vert_2 = \max_{x \neq 0} \frac{\Vert Ax \Vert_2}{\Vert x \Vert_2} = \max_{x \neq 0, \Vert x \Vert_2 = 1} \Vert Ax \Vert_2, \quad x \in \mathbb{R}^n.$$
Moreover, keep in mind that $\Vert A \Vert_2$ is also equal to $\sqrt{\lambda_{\max}}$, where $\lambda_{\max}$ denotes the largest eigenvalue of the symmetric (positive semidefinite) matrix $A^TA$ \cite[Sec.\,2.3.3]{GolubVanLoan}.
We start by considering the case where $(x_1, \ldots , x_n)$ is an ordered orthonormal basis of a Lagrangian subspace $\mathcal{L} \subset \mathbb{R}^{2n}$.

\begin{lemma} \label{thm:leastnorm1}
Let $\mathcal{L} \subset \mathbb{R}^{2n}$ be a Lagrangian subspace and $\mathbf{B} = (x_1, \ldots , x_n)$ any ordered orthonormal basis of $\mathcal{L}$. Furthermore, let $\widehat{H} \in \mathbb{R}^{2n \times 2n}$ be as defined in \eqref{equ:H_hat}. Then
$$ 1 = \Vert \widehat{H} \Vert_2 \leq \Vert H \Vert_2 \quad \forall \, H \in \mathbb{HK}(\mathbf{B}).$$
\end{lemma}

\begin{proof}
We use the fact that $\Vert \widehat{H} \Vert_2$ is equal to the largest eigenvalue of $\widehat{H}^T\widehat{H}$. From the construction of $\widehat{H}$, the properties $X_R^TX_R = X_L^TX_L = I_{n-1}$ (due to the orthogonality of the basis) and $X_R^TJ_{2n}X_L = X_L^TJ_{2n}X_R = 0$ we obtain
\begin{equation} \widehat{H}^T\widehat{H} = X_LX_L^T + J_{2n}^T X_R X_R^T J_{2n}. \label{equ:HTH} \end{equation}
Now a direct calculation shows that $(\widehat{H}^T \widehat{H})^2 = \widehat{H}^T \widehat{H}$. Therefore, $\widehat{H}^T \widehat{H}$ is a projection matrix that only has the eigenvalues zero and one (it projects onto the space $\textnormal{span} \lbrace x_1, \ldots , x_{n-1}, J_{2n}^Tx_2, \ldots , J_{2n}^Tx_n \rbrace$). Therefore it follows that $\Vert \widehat{H} \Vert_2 = 1$.
\end{proof}

The next lemma will be important to prove Theorem \ref{thm:leastnorm2} below. It relates the eigenvalues of $\widehat{H} = X_RX_L^+ + J^T(X_RX_L^+)^TJ$ to the eigenvalues of $K :=X_RX_L^+$.

\begin{lemma} \label{thm:equalnorm}
Let $\mathcal{L} \subset \mathbb{R}^{2n}$ be a Lagrangian subspace and $\mathbf{B} = (x_1, \ldots , x_n)$ any ordered basis of $\mathcal{L}$. Furthermore, let $\widehat{H} \in \mathbb{R}^{2n \times 2n}$ be as defined in \eqref{equ:H_hat} and $K = X_RX_L^+$. Then we have $$ \Vert \widehat{H} \Vert_2 = \Vert K \Vert_2.$$
\end{lemma}

\begin{proof}
We now use the fact that $\Vert \widehat{H} \Vert_2 = \sqrt{\lambda_{\max}}$ where $\lambda_{\max}$ denotes the largest eigenvalue of $\widehat{H} \widehat{H}^T$. Similarly, $\Vert K \Vert_2$ is equal to $\sqrt{\mu_{\max}}$ for the largest eigenvalue  $\mu_{\max}$ of $KK^T$. Moreover, keep in mind that $X_L^+ = (X_L^TX_L)^{-1}X_L^T$.

Now we obtain
$$KK^T = X_R \left( X_L^TX_L \right)^{-1}X_L^TX_L \left( X_L^TX_L \right)^{-1} X_R^T = X_R \left( X_L^TX_L \right)^{-1}X_R^T$$
and, computing $\widehat{H} \widehat{H}^T$ using the relations $X_L^TJ_{2n}X_R=0$ and $X_R^TJ_{2n}X_L=0$,
$$ \widehat{H} \widehat{H}^T = KK^T + J_{2n}^T \left( K^TK \right) J_{2n}.$$

Now we have a closer look at $KK^T = X_R ( X_L^TX_L )^{-1}X_R^T$. First note that $\textnormal{range}(KK^T) \subset \textnormal{range}(X_R)$ holds. Thus, if $0 \neq \mu \in \mathbb{C}$ is an eigenvalue of $KK^T$ with eigenvector $0 \neq z \in \mathbb{R}^{2n}$, then $z \in \textnormal{range}(X_R) = \textnormal{span} \lbrace x_2, \ldots , x_n \rbrace$. For $z$ and $\widehat{H} \widehat{H}^T$ we obtain
$$ \begin{aligned} \big( \widehat{H} \widehat{H}^T \big) z &= KK^Tz + J_{2n}^T X_L \left( X_L^T X_L \right)^{-1}X_R^TX_R \left( X_L^TX_L \right)^{-1}X_L^TJ_{2n} z \\
&= KK^Tz = \mu z
\end{aligned}$$
since $X_L^TJ_{2n}z = 0$. Therefore, any nonzero eigenvalue of $KK^T$ is also an eigenvalue of $\widehat{H} \widehat{H}^T$. Now we show that any eigenvalue $\lambda \neq 0$ of $\widehat{H} \widehat{H}^T$ is also an eigenvalue of $KK^T$.

First, set $X := [ \, x_1 \; \cdots \; x_n \, ]$. As any $z_1,z_2 \in \mathbb{R}^n$ with the property $X z_1 = J_{2n}^TXz_2$ are necessarily zero (multiplying the equation from the left by $X^+$ gives $z_1=0$ which implies $z_2=0$), the vectors  $x_1, \ldots , x_n, J_{2n}^Tx_1, \ldots , J_{2n}^Tx_n$ are linearly independent. In particular, they form a basis of $\mathbb{R}^{2n}$. Accordingly, if $0 \neq \lambda \in \mathbb{R}$ is an eigenvalue of $\widehat{H} \widehat{H}^T$ with  eigenvector $0 \neq z \in \mathbb{R}^{2n}$, then
\begin{equation} z = Xy + J_{2n}^TXw \label{equ:proof-z} \end{equation}
for some vectors $y=[ \, \zeta_1 \; \cdots \; \zeta_n \, ]^T, w = [\, \omega_1 \; \cdots \; \omega_n \, ]^T \in \mathbb{R}^n$. We consider two special cases first:
\begin{itemize}
\item[(a)] First assume $y = 0$, i.e.\ $z = J_{2n}^TXw$. Then
$$\lambda z = \widehat{H} \widehat{H}^Tz = \left( KK^T + J_{2n}^T \left( K^TK \right) J_{2n}\right) z = \left( J_{2n}^T \left( K^TK \right) J_{2n}\right) z$$
since $ KK^Tz = X_R ( X_L^TX_L )^{-1}X_R^TJ_{2n}^TXw = 0$ because $X_R^TJ_{2n}X = 0$. Therefore, $(\lambda, z)$ is an eigenpair of $J_{2n}^T (K^TK) J_{2n}$. As $KK^T$ and $K^TK$ have the same nonzero eigenvalues and since $J_{2n}^T (K^TK) J_{2n}$ is a similarity transformation of $K^TK$, this proves that $\lambda$ is an eigenvalue of $KK^T$.
\item[(b)] Now assume $w = 0$, i.e.\ $z=Xy$. Then
$$\lambda z = \widehat{H} \widehat{H}^Tz =  \left( KK^T + J_{2n}^T \left( K^TK \right) J_{2n}\right) z = KK^Tz$$
since $J_{2n}^T X_L \big( X_L^T X_L \big)^{-1}X_R^TX_R \big(X_L^TX_L \big)^{-1}X_L^TJ_{2n} Xz=0$ (which follows because $X_L^TJ_{2n}X=0$).  Thus, $\lambda$ is also an eigenvalue of $KK^T$.
\end{itemize}
Now consider the case $y \neq 0$ and $w \neq 0$ for $z$ in \eqref{equ:proof-z}. It follows directly from (a) and (b) above that
\begin{equation} \lambda z = \widehat{H} \widehat{H}^Tz = KK^TXy + \left( J_{2n}^T \left( K^TK \right) J_{2n}\right) J_{2n}^TXw. \label{equ:eig-equ} \end{equation}
Recall that $\textnormal{range}(KK^T) \subset \textnormal{span} \lbrace x_2, \ldots , x_n \rbrace$ while $\textnormal{range}(J_{2n}^T(K^TK)J_{2n}) \subset \textnormal{span} \lbrace J_{2n}^Tx_1, \ldots , J_{2n}^Tx_{n-1} \rbrace$ since
\begin{equation} K^TK = X_L \left( X_L^TX_L \right)^{-1}X_R^TX_R \left( X_L^TX_L \right)^{-1} X_L^T. \label{range_KTK} \end{equation}
In particular, $\lbrace x_1, J_{2n}^Tx_n \rbrace \notin \textnormal{range}(\widehat{H} \widehat{H}^T)$. Thus, for \eqref{equ:eig-equ} to hold it necessarily follows that $\zeta_1 = 0$ and $\omega_n = 0$. Consequently, $Xy \in \textnormal{range}(X_R)$ while $J_{2n}^TXw \in \textnormal{range}(J_{2n}^TX_L)$. Since
$$ \textnormal{range}(X_R) \cap \textnormal{range}(J_{2n}^TX_L) = \lbrace 0 \rbrace$$
it follows again from \eqref{equ:eig-equ} that
\begin{align*}
\lambda Xy = \left(KK^T \right) Xy \quad \textnormal{and} \quad \lambda J_{2n}^TXw = \left( J_{2n}^T K^TK J_{2n} \right) J_{2n}^TXw
\end{align*}
have to hold. These relations both imply that $\lambda$ is an eigenvalue of $KK^T$. In conclusion, any nonzero eigenvalue of $\widehat{H} \widehat{H}^T$ is also an eigenvalue of $KK^T$ and so the nonzero eigenvalues of $KK^T$ and $\widehat{H} \widehat{H}^T$ coincide. In particular, $\lambda_{\max} = \mu_{\max}$ and so $\Vert \widehat{H} \Vert_2 = \Vert K \Vert_2$ follows.
\end{proof}

Before we state the analogous result to Theorem \ref{thm:leastnorm} for the $\Vert \cdot \Vert_2$-norm, we need some observations. To this end, let $K = X_RX_L^+ = X_R (X_L^TX_L)^{-1}X_L^T$ for some ordered basis $(x_1, \ldots , x_n)$ of a given Lagrangian subspace $\mathcal{L} \subset \mathbb{R}^{2n}$ as before (using the definitions from \eqref{equ:XLXR}). First notice that
\begin{equation}
\Vert K \Vert_2^2 = \max_{z \in \mathbb{R}^{2n}, \, \Vert z \Vert_2 = 1} \Vert Kz \Vert_2^2 = \max_{z \in \mathbb{R}^{2n}, \, \Vert z \Vert_2 = 1} z^TK^TKz.
\label{equ:courantfischer}
\end{equation}
As a consequence of the Courant-Fischer-Theorem, the maximum on the right-hand-side of \eqref{equ:courantfischer} is attained for $\tilde{z} \in \mathbb{R}^{2n}$ if $\tilde{z}$ (with $\Vert \tilde{z} \Vert_2 = 1$) is an eigenvector for $K^TK$ for its largest eigenvalue $\mu_{\max} > 0$. Then $\Vert K \Vert_2 = \Vert K \tilde{z} \Vert_2$. From $(K^TK) \tilde{z} = \mu_{\max} \tilde{z}$ it trivially follows that $\tilde{z} \in \textnormal{range}(K^TK)$. Moreover
$\textnormal{range}(K^TK) \subset \textnormal{range}(X_L)$ holds as can be seen from \eqref{range_KTK}. Thus, $\tilde{z} \in \textnormal{range}(X_L)$ and we conclude that
\begin{equation}
\Vert K \Vert_2 = \max_{z \in \mathbb{R}^{2n}, \, \Vert z \Vert_2 = 1} \Vert Kz \Vert_2 = \max_{z \in \textnormal{range}(X_L), \, \Vert z \Vert_2 = 1} \Vert Kz \Vert_2.
\label{equ:max-range}
\end{equation}

With \eqref{equ:max-range} at hand we may now easily prove the following theorem.

\begin{theorem} \label{thm:leastnorm2}
Let $\mathcal{L} \subset \mathbb{R}^{2n}$ be a Lagrangian subspace and $\mathbf{B} = (x_1, \ldots , x_n)$ some ordered basis of $\mathcal{L}$. Furthermore, let $\widehat{H} \in \mathbb{R}^{2n \times 2n}$ be as defined in \eqref{equ:H_hat}. Then
$$ \Vert \widehat{H} \Vert_2 \leq \Vert H \Vert_2 \qquad \forall \, H \in \mathbb{HK}(\mathbf{B}).$$
\end{theorem}

\begin{proof}
Let $H = \widehat{H} - \Delta H \in \mathbb{HK}(\mathbf{B})$ with $\Delta H \in \mathcal{H}_0$ and $\widehat{H} = K + K^\star$ as in \eqref{equ:H_hat}. Then, as $K^\star x_k = 0$  and $(\Delta H)x_k = 0$ hold for all $k=1, \ldots , n-1$, it follows that $H = \widehat{H}-\Delta H$, $\widehat{H}$ and $K$ all behave exactly identically on the subspace $\textnormal{span} \lbrace x_1, \ldots , x_{n-1} \rbrace = \textnormal{range}(X_L)$. That means, for any $y \in \textnormal{range}(X_L)$ we have $Hy = \widehat{H}y = Ky$. Using the result from Lemma \ref{thm:equalnorm} along with the observation in \eqref{equ:max-range} we can estimate
\begin{align*}
\Vert \widehat{H} \Vert_2 &= \Vert K \Vert_2 = \max_{z \in \textnormal{range}(X_L), \, \Vert z \Vert_2 = 1} \Vert Kz \Vert_2 = \max_{z \in \textnormal{range}(X_L), \, \Vert z \Vert_2 = 1} \Vert Hz \Vert_2 \\ &\leq \max_{z \in \mathbb{R}^{2n}, \, \Vert z \Vert_2 = 1} \Vert Hz \Vert_2 = \Vert H \Vert_2.
\end{align*}
Thus, the lower bound $\Vert \widehat{H} \Vert_2 \leq \Vert H \Vert_2$ holds for all skew-Hamiltonian matrices $H \in \mathbb{R}^{2n \times 2n}$ from $\mathbb{HK}(\mathbf{B})$.
\end{proof}

\subsection{Matrices in $\mathbb{HK}(\mathbf{B})$ with prescribed eigenvalues of $H|_{\mathcal{L}}$}

Let $\Lambda = \lbrace \lambda_1, \ldots , \lambda_n \rbrace \subset \mathbb{C}$ be a given set of $n$ scalars $\lambda_k \in \mathbb{C}$ closed under conjugation. Corresponding to $\Lambda$ we define
\begin{equation} p(x) := \prod_{k=1}^n (x- \lambda_k) = x^n + a_{n-1}x^{n-1} + \cdots + a_1x + a_0 \in \mathbb{R}[x]. \label{equ:polynomial} \end{equation}
Let $\mathcal{L} \subset \mathbb{R}^{2n}$ be a Lagrangian subspace with basis $\mathbf{B}=(x_1, \ldots , x_n)$ and let $X = [ \, x_1 \; \cdots \; x_n \, ] \in \mathbb{R}^{2n \times n}$. If $H \in \mathbb{HK}(\mathbf{B})$, then the eigenvalues of $H|_{\mathcal{L}}$ are $\lambda_1, \ldots , \lambda_n$ if and only if
\begin{equation} H \begin{bmatrix} x_1 & \cdots & x_n \end{bmatrix} = \begin{bmatrix} x_1 & \cdots & x_n \end{bmatrix} \begin{bmatrix} 0 & \cdots & 0 & -a_0 \\ 1 & & \vdots & -a_1 \\ & \ddots & 0 & \vdots \\ 0 & & 1 & -a_{n-1} \end{bmatrix} =: XC. \label{equ:companion_B} \end{equation}
The matrix $C \in \mathbb{R}^{n \times n}$ is called the companion matrix for $p(x)$. Next, consider the skew-Hamiltonian matrix  $\widetilde{H} := XCX^+ + J_{2n}^T(XCX^+)^TJ_{2n}$.
A direct calculation shows that
$$ \widetilde{H}X = XC(X^+X) + J_{2n}^T(X^+)^TC^T(X^TJ_{2n}X) = XC,$$
so the eigenvalues of $H$ restricted to $\mathcal{L}$ are exactly $\lambda_1, \ldots , \lambda_n$. Moreover, $\widetilde{H} \in \mathbb{HK}(\mathbf{B})$ since
$$ \begin{aligned} \widetilde{H}x_k &= (XCX^+)x_k + J_{2n}^T(X^+)^TC^TX^TJ_{2n}x_k \\ &= (XCX^+)x_k = XCe_k = Xe_{k+1} = x_{k+1} \end{aligned} $$
holds for all $k=1, \ldots , n-1$. We obtain the following theorem characterizing all matrices in $\mathbb{HK}(x_1, \ldots , x_n)$ that give a prescribed set of eigenvalue when restricted to the Lagrangian subspace $\mathcal{L}= \textnormal{span}\lbrace x_1, \ldots , x_n \rbrace$.

\begin{theorem} \label{thm:eigenvalues}
Let $\mathcal{L} \subset \mathbb{R}^{2n}$ be a Lagrangian subspace and $\mathbf{B} = (x_1, \ldots , x_n)$ some ordered basis of $\mathcal{L}$. Furthermore, let $\lbrace \lambda_1, \ldots , \lambda_n \rbrace \subset \mathbb{C}$ be a set of $n$ complex scalars that is closed under conjugation. Let $p(x) \in \mathbb{R}[x]$ be defined as in \eqref{equ:polynomial}
and suppose $C \in \mathbb{R}^{n \times n}$ as given in \eqref{equ:companion_B} is the companion matrix for $p(x)$. Then any skew-Hamiltonian matrix $H \in \mathbb{HK}(\mathbf{B})$ such that the eigenvalues of $H|_{\mathcal{L}}$ are given by $\lambda_1, \ldots , \lambda_n$ can be expressed as
\begin{equation}  H = \widetilde{H} + J_{2n}^TX^{\bot}S(X^{\bot})^T \label{equ:H_eigvals} \end{equation}
where $\widetilde{H} = XCX^+ + J_{2n}^T(XCX^+)^TJ_{2n}\in \mathbb{R}^{2n \times 2n}$, $X^{\bot} \in \mathbb{R}^{2n \times n}$ is a matrix whose columns are an orthonormal basis for $\mathcal{L}^{\bot}$ and $S \in \mathbb{R}^{n \times n}$ is some skew-symmetric matrix.
\end{theorem}

\begin{proof}
Let $X = [\, x_1 \; \cdots \; x_n \, ]$. It is clear that any matrix $H$ of the form \eqref{equ:H_eigvals} satisfies $HX = XC$. Therefore, $H \in \mathbb{HK}(\mathbf{B})$ and the eigenvalues of $H|_{\mathcal{L}}$ coincide with the eigenvalues of $C$ which are $\lambda_1, \ldots , \lambda_n$. Now suppose $H \in \mathbb{R}^{2n \times 2n}$ is skew-Hamiltonian and satisfies $HX = XC$. Then $(\widetilde{H}-H)X=0$, so $\widetilde{H}-H$ is a skew-Hamiltonian matrix with $\mathcal{L} \subset \textnormal{null}(\widetilde{H}-H)$. With the same reasoning as used for the proof of Theorem \ref{thm:leastnorm} we find that any skew-Hamiltonian matrix with this property can be expressed as $J_{2n}^TX^{\bot}S(X^{\bot})^T$ where the columns of $X^{\bot}$ form an orthonormal basis of $\mathcal{L}^{\bot}$ and $S=-S^T \in \mathbb{R}^{n \times n}$ is a suitable skew-symmetric matrix. Thus $\widetilde{H}-H = J_{2n}^TX^{\bot}S(X^{\bot})^T$ for a suitable $S$ implies that $H$ has the form in \eqref{equ:H_eigvals} and completes the proof.
\end{proof}

\section{Applications}
\label{sec:applications}

In this section we present a possible application for the results obtained in the previous section.
First, suppose $A \in \mathbb{R}^{2n \times 2n}$ is a (slightly) perturbed skew-Hamiltonian matrix, that is, $A \neq A^\star$ but $\Vert A - A^\star \Vert_F$ is small. Moreover, let $x_1 \in \mathbb{R}^{2n}$ be some fixed vector. Now consider $A$ and its Krylov sequence
$$Ax_k = x_{k+1}, \quad k=1, \ldots , n-1.$$
Let $\widetilde{\mathcal{L}} = \mathcal{K}_n(A,x_1)$ be the corresponding Krylov space. Since $A$ is not skew-Hamiltonian, $\widetilde{\mathcal{L}}$ will in general neither be isotropic nor $A$-invariant as it would have been the case if $A$ were skew-Hamiltonian. Assume that $\widetilde{\mathcal{L}}$ is $n$-dimensional. Since $\Vert A-A^\star \Vert_F$ was assumed to be small, it is a reasonable assumption that the distance\footnote{By distance we mean the gap between $\mathcal{L}$ and $\widetilde{\mathcal{L}}$ measured as $\Vert P_{\mathcal{L}} - P_{\widetilde{\mathcal{L}}} \Vert_2$ where $P_{\mathcal{L}}$ and  $P_{\widetilde{\mathcal{L}}}$ are the orthogonal projectors onto these subspaces.} between $\widetilde{\mathcal{L}}$ and a true Lagrangian subspace $\mathcal{L}$ is not too large.
We may find an approximation of $\mathcal{L}$ and a basis $\mathbf{B}=(y_1, \ldots , y_n)$ of $\mathcal{L}$ by the following easy procedure: \\[0.5cm]
Set $y_1 = x_1$ and $Y = [ \, y_1 \, ]$. \\
For each $k=2, \ldots , n$ \\
\hspace*{1cm} Compute the orthogonal projection $y_k$ of $x_k$ onto the nullspace of \\ \hspace*{1cm} $Y^TJ_{2n}$. \\
\hspace*{1cm} Set $Y = [\, Y \; y_k \, ]$. \\
End \\

Suppose that $y_1, \ldots , y_n \in \mathbb{R}^{2n}$ are linearly independent. In this case, they span a Lagrangian subspace.
Using $\mathbf{B}=(y_1, \ldots , y_n)$ as an approximation for the Krylov sequence of the unperturbed matrix we may now ask for a skew-Hamiltonian matrix $H \in \mathbb{R}^{2n \times 2n}$ with $Hy_k = y_{k+1},$ $k=1, \ldots , n-1$, as close as possible to $A$. The same approach as in the proof of Theorem \ref{thm:leastnorm} gives the following result.

\begin{lemma} \label{lem:bestapprox1}
Let $A \in \mathbb{R}^{2n \times 2n}$ and let $\mathbf{B}=(x_1, \ldots , x_n)$ be the basis of some Lagrangian subspace $\mathcal{L} \subset \mathbb{R}^{2n}$. Then
$$ \Vert H - A \Vert_F \leq \Vert H' - A \Vert_F \quad \forall \; H' \in \mathbb{HK}(\mathbf{B})$$
if $H$ is of the form \eqref{equ:char_HK} with
\begin{equation} S = - \frac{1}{2} \left( (X_L^{\bot})^TJ_{2n}(\widehat{H}-A)X_L^{\bot} + (X_L^{\bot})^T(\widehat{H}-A)^TJ_{2n}X_L^{\bot} \right). \label{equ:best_S1} \end{equation}
\end{lemma}

\begin{proof}
If $H \in \mathbb{HK}(\mathbf{B})$, then
$ \Vert H - A \Vert_F = \Vert (\widehat{H} - A) + J_{2n}^TX_L^{\bot} S (X_L^{\bot})^T \Vert_F$
which is to be minimized for $S$. Notice that the solution
\begin{equation} S' = (X_L^{\bot})^TJ_{2n}(\widehat{H}-A) X_L^{\bot}  \label{equ:best_S2} \end{equation}
we obtained for $A=0$ in the derivations \eqref{equ:minimizeS} is not valid here since $S' \neq -(S')^T$ might hold. As the other three blocks in \eqref{equ:minimizeS} are not effected by the choice of $S$, any matrix $S=-S^T$ that yields a best possible approximation in the Frobenius norm is obtained when $S$ is chosen as close as possible to the matrix $S'$ in \eqref{equ:best_S2}. According to \cite[Sec.\,2]{High89} one such solution is given by $S = (1/2)(S'-(S')^T)$ which is exactly the matrix in \eqref{equ:best_S1}.
\end{proof}

In order to test Lemma \ref{lem:bestapprox1} we first create an orthornormal basis $\mathbb{B}'=(x'_1, \ldots , x'_n)$ for a Lagrangian subspace $\mathcal{L}'\subset \mathbb{R}^{2n}$ with the isotropic Arnoldi algorithm from \cite{MehrWat}. With this basis and \eqref{equ:H_hat} a skew-Hamiltonian matrix $\hat{H}$ is build which realizes $\mathcal{K}_n(\hat{H},x'_1)$. We perturb this matrix by a normally distributed random matrix $E$ with magnitude $\beta$, e.g.,
\begin{align}
    A = \hat{H} + \beta E,\nonumber
\end{align}
where $\beta = 0.001$. The matrix $A$ is constructed that way to get a not too ill conditioned Krylov sequence which we calculate explicitly from $A$ and $x_1$ to obtain $\widetilde{\mathcal{L}} = \mathcal{K}_n(A,x_1)$. From $\widetilde{\mathcal{L}}$ we get $\mathcal{L}$ and $\mathbb{B}=(y_1,\ldots,y_n)$ by the above-mentioned procedure. We now use Lemma \ref{lem:bestapprox1} to construct the skew-Hamiltonian matrix $H\in \mathbb{R}^{2n \times 2n}$ with $Hy_k = y_{k+1},$ $k=1, \ldots , n-1$ and as close as possible to $A$.
In Figure \ref{fig:SkewHamiltonicity_Nearness_Isotropy} we show the skew-Hamiltonian property of $A$ and $H$, how close $H$ to $A$ is and the isotropy of $X=[x_1\;\ldots\;x_n]$ and $Y=[y_1\;\ldots\;y_n]$. It can be seen that our approach works quite accurate. 
The relative distance between $H$ and $A$ grows not that fast than the loss of the isotropy property of $\mathcal{K}_n(A,x_1)$. We observe that the gap between $\widetilde{\mathcal{L}}$ and $\mathcal{L}$ is approximate of order $\mathcal{O}(n\beta)$. 
The numerical experiment was performed with MATLAB Version 9.6.0.1114505 (R2019a) Update 2.
\begin{figure}
    \centering
    \includegraphics
    [width=1\textwidth]
    {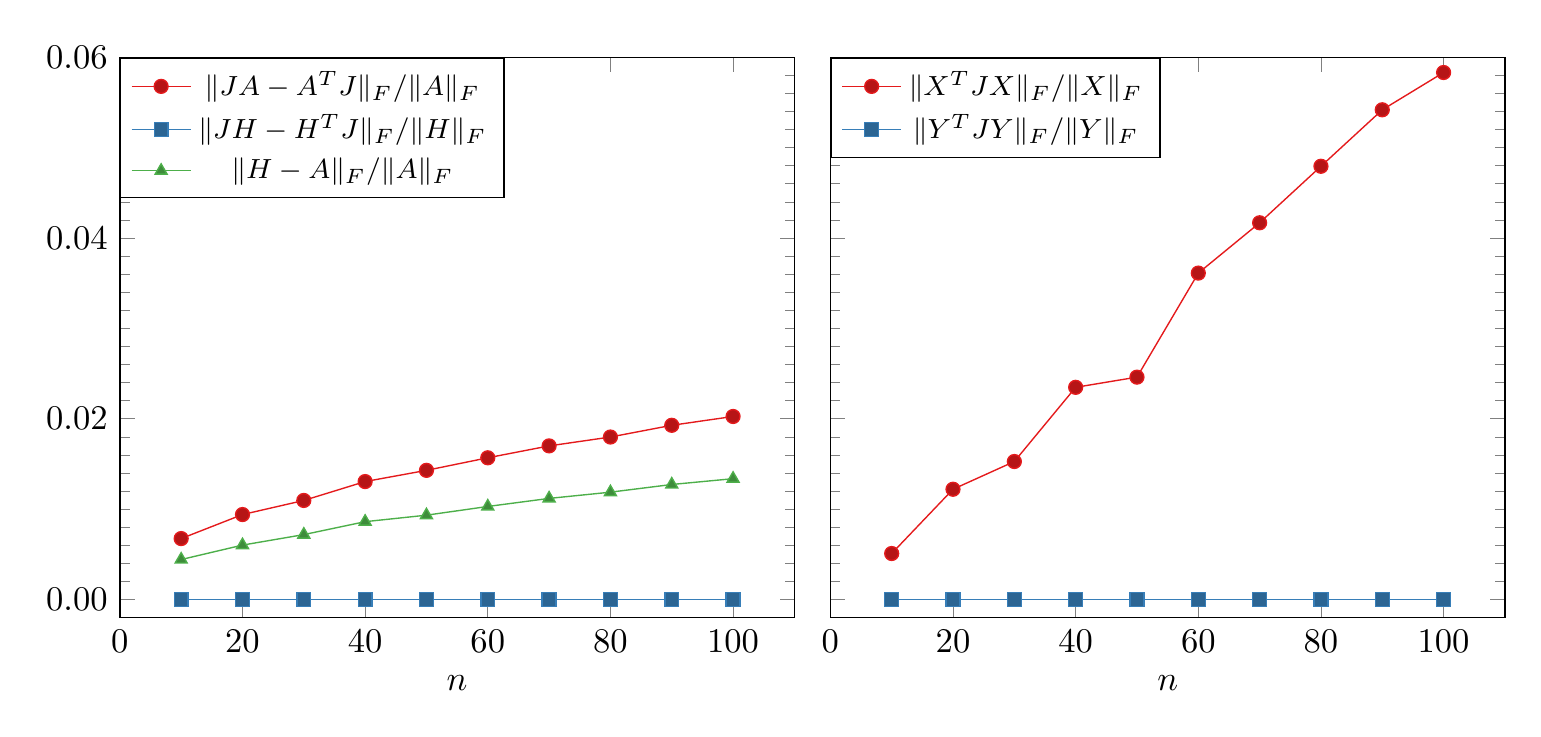}
    \caption[]{(\textit{left part}) The relative closeness of $H$ to $A$ and the relative variation from the skew-Hamiltonianicity of these two matrices is displayed.\\
    (\textit{right part}) The relative deviation from the isotropy property of the Krylov spaces $\mathcal{K}_n(A,x_1)$ and $\mathcal{K}_n(H,y_1)$ is measured.}
    \label{fig:SkewHamiltonicity_Nearness_Isotropy}
\end{figure}

Assuming $\mathbf{B}=(x_1, \ldots , x_n)$ is a basis of a Lagrangian subspace $\mathcal{L} \subset \mathbb{R}^{2n}$ we may also determine the matrix $H \in \mathbb{HK}(\mathbf{B})$ as close as possible to a given matrix $A \in \mathbb{R}^{2n \times 2n}$ such that the eigenvalues of $H|_{\mathcal{L}}$ coincide with a predetermined set of $n$ numbers $\lambda_1, \ldots , \lambda_n$. In fact, to find such a matrix $H$, the proof of Lemma \ref{lem:bestapprox1} can be carried out analogously starting with the characterization given in \eqref{equ:H_eigvals}. 

\section{Conclusions}
\label{sec:conclusions}

In this work we characterized the set of all skew-Hamiltonian matrices $H$ for which a given isotropic subspace $\mathcal{L}$ arises as a Krylov space. That is, given a basis $\mathbf{B}=(x_1, \ldots, x_n)$ of some Lagrangian subspace $\mathcal{L} \subset \mathbb{R}^{2n}$, we analyzed the set $\mathbb{HK}(\mathbf{B})$ of all skew-Hamiltonian matrices $H \in \mathbb{R}^{2n \times 2n}$ that satisfy $Hx_k = x_{k+1}$ for $k=1, \ldots , n-1$. We identified elements $H \in \mathbb{HK}(\mathbf{B})$ with minimal 2-norm and minimal Frobenius. Moreover, we characterized all matrices $H \in \mathbb{HK}(\mathbf{B})$ such that $H|_{\mathcal{L}}$ has $n$ predetermined eigenvalues. Finally, we analyzed a scenrio where these results can be applied.

\section*{Acknowledgements}
We are grateful to the anonymous referee whose report and suggestions helped to improve the manuscript significantly. Moreover, we would like to thank our colleague Christian Bertram for pointing out the connection to DMD to us.

\end{document}